\documentclass[preprint, 12pt, 3p]{elsarticle}

\usepackage{amssymb}

\usepackage{amsmath, amsthm, amssymb,amsfonts,graphicx}
\usepackage{enumitem}

\newcommand*\df{\mathop{}\!\mathrm{d}}

\newcommand{\bsx}{\boldsymbol{x}}
\newcommand{\bst}{\boldsymbol{t}}
\newcommand{\bsl}{\boldsymbol{\ell}}
\newcommand{\bsk}{\boldsymbol{k}}
\newcommand{\bsy}{\boldsymbol{y}}

\newtheorem{theorem}{Theorem}

\newtheorem{lemma}{Lemma}

\theoremstyle{remark}
\newtheorem{remark}{Remark}

\journal{Journal of Complexity}

\begin{document}

\begin{frontmatter}

\title{Quasi-Monte Carlo tractability of high dimensional integration over products of simplices}

\author{Kinjal Basu\corref{cor1}}
\address{Department of Statistics \\ Stanford University}

\cortext[cor1]{Corresponding author}
\ead{kinjal@stanford.edu}

\begin{abstract}
Quasi-Monte Carlo (QMC) methods for high dimensional integrals over unit cubes and products of spheres are well-studied in literature. We study QMC tractability of integrals of functions defined over the product of $m$ copies of the simplex $T^d \subset \mathbb{R}^{d}$. The domain is a tensor product of $m$ reproducing kernel Hilbert spaces defined by `weights' $\gamma_{m,j}$, for $j = 1,2, \ldots, m$. Similar to the results on the unit cube in $m$ dimensions, and the product of $m$ copies of the $d$-dimensional sphere, we prove that strong polynomial tractability holds iff $\limsup_{m \rightarrow \infty} \sum_{j=1}^m \gamma_{m,j} < \infty$ and polynomial tractability holds iff $\limsup_{m \rightarrow \infty} \frac{\sum_{j=1}^m \gamma_{m,j}}{\log(m + 1 )}  < \infty$. We also show that weak tractability holds iff $\lim_{m \rightarrow \infty} \frac{\sum_{j=1}^m \gamma_{m,j}}{m} = 0$. The proofs employ Sobolev space techniques and weighted reproducing kernel Hilbert space techniques for the simplex and products of simplices as domain. Properties of orthogonal polynomials on a simplex are also used extensively.
\end{abstract}

\begin{keyword}
Quasi-Monte Carlo methods \sep Multivariate integration \sep Product of simplices \sep Worst-case error \sep Tractability
\end{keyword}

\end{frontmatter}


\section{Introduction}
\label{sec1}
Integration over simplices is an important problem in computer graphics and light transport theory. This applies, in particular, to image rendering. The main problem is to calculate an integral of the form
\begin{equation}
\int_{T^d} \cdots \int_{T^d} f(x_1, \ldots, x_m) \df x_1 \df x_2 \cdots \df x_m,  
\end{equation}
where $d$  is usually small, $m$ can be very large and $T^d$ is the $d$-dimensional simplex defined by 
\[
T^d := \left\{ x \in \mathbb{R}^d : x_1 \ge 0, \ldots, x_d \ge 0, 1 - \sum_{i=1}^d x_i \ge 0 \right\}.
\]

Sampling over such product spaces is quite challenging. Most quasi-Monte Carlo (QMC) techniques are well developed for numerical integration of functions defined on the unit cube $[0,1]^m$.
The quantity $\mu := \int_{[0,1]^m} f(\bsx) \df \bsx  $ is
approximated by an equal weight rule
$\hat\mu_N := (1/N)\sum_{i=1}^N f(\bsx_i)$ for carefully chosen $\bsx_i\in[0,1]^m$.
The accuracy of such a QMC method can be measured using the Koksma-Hlawka inequality (see \cite{Niederreiter1992}) which suggests to use low-discrepancy sequences in order to reduce the error of numerical integration. However, they have a cost (in terms of the number of function evaluations) which grows exponentially with $m$. In this setting of the unit cube, Sloan and Wo\'{z}niakowski \cite{Sloan1998} find a class of functions for which the cost is bounded independently of $m$. (They use $d$ instead of $m$ for the dimension). Later on Kuo and Sloan \cite{Kuo2005} study numerical integration over products of unit spheres (in a fixed dimensional Euclidean space) and find a class of functions with the same property. In this paper, we study such a class of functions by considering the domain to be products of standard simplices.


	
Interest in numerical integration over the simplex is quite recent. Brandolini et al. \cite{Brandolini2013} develop a new Koksma-Hlawka type inequality on the simplex to study low-discrepancy sequences. Pillards and Cools \cite{Pillards2005} give a series of transformations from the unit cube to the simplex and very recently, Basu and Owen \cite{Basu2014} give two explicit low-discrepancy constructions on the triangle. 

Similar to \cite{Kuo2005, Sloan1998} we find a function class $H_m$, such that the number of function evaluations needed to reduce the initial error by a factor of $\epsilon$ is bounded independently of $m$. We also prove that error in numerical integration is $O(n^{-1/2})$  for functions belonging to the class $H_m$. We use the well-known, elegant and powerful method of reproducing kernel Hilbert spaces in the weighted case by combining it with tensor-product techniques. The smoothness condition of $r > d/2$ for the underlying space in the case of the sphere, ensures that these techniques work (see \cite{Kuo2005}). In our situation for the embedding results, point evaluations and the well-definiteness of worst-case error, we need $r > d+1$.  Here $r$ is the number of derivatives used in defining the inner product associated with the function class $H_m$. Our results are based on finding an orthonormal basis on the simplex which has similar properties to that of spherical harmonics on the sphere. A detailed overview of the connections between the sphere, ball and simplex can be found in \cite{xu2006analysis}. 

The rest of the paper is organized as follows. We describe the problem in Section \ref{sec2}. The preliminaries are explained in Section \ref{sec3}. We prove our theorem on QMC tractability in Section \ref{sec4}. Concluding remarks and discussions follow in Section \ref{sec5}.

\section{Numerical integration over products of simplices and tractability}
\label{sec2}
Let $T^d$ be the $d$-dimensional simplex defined by 
\begin{equation}
\label{td}
T^d := \{ x \in \mathbb{R}^d : x_1 \ge 0, \ldots, x_d \ge 0, 1 - |x| \ge 0 \},
\end{equation}
where $|x| := x_1 + \cdots + x_d.$ The purpose of this paper is to study the problem of integration on the product space $(T^d)^m := T^d \times T^d \times \cdots \times T^d$ of $m$ copies of the $d$-dimensional simplex $T^d$. Usually, $d$ is small and $m$ is very large. There are many papers on integration over the triangle $T^2$ and the simplex \cite{ Basu2014, pillards2004theoretical,  Pillards2005} but very little is known when the domain is the product of higher dimensional simplices. 

Let us begin with a few notations, which we use throughout the rest of the paper. We write the integral as 
\begin{equation}
\label{int}
I_m(f) := \frac{1}{c_m} \int_{(T^d)^m} f(\bsx) \df \bsx, 
\end{equation}
where $c_m  := \int_{(T^d)^m} 1 \df \bsx = 1/(d!)^m$ is the normalizing constant. The integrand is assumed to belong to some Sobolev space $H_m$ which is a tensor product of $m$ weighted reproducing kernel Hilbert spaces (RKHS), where the $j$th weighted RKHS is parametrized by weights $\gamma_{m,j}$ for $j = 1, \ldots, m$. We give the explicit definition in Section \ref{sec3}. 

We approximate the integral (\ref{int}) using the following QMC method:
\begin{equation}
\label{qmc}
Q_{n,m}(f) := \frac{1}{n} \sum_{i=1}^n f (\bst_i) = \frac{1}{n} \sum_{i=1}^n f(t_{i,1},\ldots, t_{i,m}),
\end{equation}
where $\bst_1, \ldots \bst_n \in (T^d)^m$. Following Sloan and Wo\'{z}niakowski \cite{Sloan1998}, we consider the worst-case error of $Q_{n,m}$ which is the worst-case performance of $Q_{n,m,}$ over the unit ball of $H_m$; i.e.,

\begin{equation}
\label{defenm}
e_{n,m} := e(Q_{n,m}) = \sup_{f \in H_m, ||f||_m \leq 1} |I_m(f) - Q_{n,m}(f)|,
\end{equation}
where $|| \cdot ||_m$ denotes the norm in $H_m$. For $n = 0$, we formally set $Q_{0,m} := 0$. The corresponding worst-case error is the initial error 
\[e_{0,m} := \sup_{f \in H_m, ||f||_m \leq 1} |I_m(f)|.
\]

We would like to reduce the initial error by a factor of $\epsilon$, where $\epsilon \in (0,1)$. Thus, we are looking for the smallest $n = n(\epsilon, m)$ for which $\bst_1, \ldots, \bst_n$ exist such that
$e_{n,m} \leq \epsilon e_{0,m}$. We can now define what we mean by QMC tractability. The general notion of tractability can be found in \cite{novak1997tractability, wozniakowski1994tractability, wozniakowski1994tractability2}. For a detailed account of tractability of multivariate problems, we refer the reader to the trilogy by Novak and Wo{\'z}niakowski \cite{novak2008tractability, novak2010tractability,novak2012tractability}. The integration problem (in the worst-case setting) is said to be `polynomial tractable' iff there exist non-negative $C, q,$ and $p$ such that 
\begin{equation}
\label{tract}
n(\epsilon,m) \le C \epsilon^{-p} m^q \qquad \forall m = 1,2, \ldots; \;\;\forall \epsilon \in (0,1). 
\end{equation}

If (\ref{tract}) holds, then the infima of $q$ and $p$ are called the $m$-exponent and $\epsilon$-exponent of tractability. The problem is said to be `strongly polynomial tractable' if (\ref{tract}) holds with $q=0$. Some years ago, a third relevant notion of tractability was introduced, namely `weak tractability' (see \cite{novak2010tractability} for details). The integration problem is said to be weakly tractable iff
\begin{align}
\lim_{\epsilon^{-1} + m \rightarrow \infty} \frac{\log n(\epsilon, m)}{\epsilon^{-1} + m} = 0.
\end{align} 

  In Section \ref{sec4}, we prove our results on the necessary and sufficient conditions for polynomial, strong polynomial and weak tractability. We will prove that if $\gamma_{m,j}$ are positive and uniformly bounded, then strong polynomial tractability holds iff
\[ \limsup_{m \rightarrow \infty} \sum_{j=1}^m \gamma_{m,j} < \infty,
\]
polynomial tractability holds iff 
\[ \limsup_{m \rightarrow \infty} \frac{\sum_{j=1}^m \gamma_{m,j}}{\log (m + 1)} < \infty
\]
and weak tractability holds iff
\[
\lim_{m \rightarrow \infty} \frac{\sum_{j=1}^m \gamma_{m,j}}{m} = 0.
\]
The conditions are exactly the same when the products of $m$ simplices and one unit cube of dimension $m$ are considered \cite{hickernell2001tractability, Sloan1998, sloan2002tractability}. It is also same when products of simplices and products of spheres of fixed dimension are considered. Kuo and Sloan \cite{Kuo2005} do not show the result for weak tractability but it easily follows from their proof. The tractability results in the case for product of simplices follow from non-constructive arguments. We are yet to develop QMC methods that achieve the lower bound of $O(n^{-1/2})$ like that in \cite{hesse2007component, kuo2002component, sloan2002step, sloan2002component, sloan2002constructing}.

\section{Preliminaries}
\label{sec3}

\subsection{The $d$-dimensional Simplex and Orthogonal Polynomials}
\label{sec3_1}
Let $L_2(T^d)$ be the space of square-integrable and measurable real-valued functions on $T^d$ provided with the inner product,
\begin{align}
\label{innerproduct}
 \langle f,g \rangle_{L_2(T^d)} = d!\int_{T^d} f(x) g(x) \df x.
\end{align}
The orthogonal polynomials with respect to this inner product have been studied extensively (see \cite{dunkl2001orthogonal}). Let $\mathcal{V}_\ell^d$ denote the space of orthogonal polynomials of degree $\ell$ with respect to this inner product.  Let $r_\ell^d = \dim \mathcal{V}_\ell^d$. It is well known that $r_\ell^d = \binom{\ell+d-1}{\ell}$. Akta\c{s} and Xu \cite{Aktas2012} obtained the following result regarding a basis of $\mathcal{V}_\ell^d$.

\begin{lemma}
For $\tilde{n} \in \mathbb{N}_0^d$ and $x \in \mathbb{R}^d$, one defines
\begin{equation}
P_{\tilde{n}}(x) := \frac{\partial^{|\tilde{n}|}}{\partial x^{\tilde{n}}} \left[ x^{\tilde{n}} (1 - |x|)^{|\tilde{n}|} \right],
\end{equation}
where $\frac{\partial^{|\tilde{n}|}}{\partial x^{\tilde{n}}} = \frac{\partial^{|\tilde{n}|}}{\partial x_1^{\tilde{n}_1} \ldots \partial x_d^{\tilde{n}_d}}$ and $x^{\tilde{n}} = x_1^{\tilde{n}_1} \ldots x_d^{\tilde{n}_d}$. $P_{\tilde{n}}$ are orthogonal polynomials with respect to constant function and $\{ P_{\tilde{n}} : |\tilde{n}| = \ell\}$ is a basis of $\mathcal{V}_\ell^d$.
\end{lemma}
For the proof of this lemma, see \cite{dunkl2001orthogonal}. Note that the set $B_{\ell} = \{ P_{\tilde{n}} : |\tilde{n}| = \ell\}$ has cardinality $r_\ell^d$ and we can order the elements of $B_{\ell}$ with some fixed ordering parametrized by $k$, where $1 \leq k \leq r_\ell^d$. Further, we can convert this basis into an orthonormal basis. For simplicity, let us denote the orthonormal basis as $\{P_{\ell,k}, 1 \leq k \leq r_\ell^d\}$. We shall now discuss a few properties of this orthonormal basis which we will use in our proofs in later sections. For more details, we refer to \cite{Aktas2012}.

\subsubsection {Orthonormality}
By orthonormality of $\{P_{\ell,k}\}$ we mean 
\[ \langle P_{\ell_1,k}, P_{\ell_2,j}  \rangle = d! \int_{T^d} P_{\ell_1,k}(x) P_{\ell_2,j}(x) \df x = \delta_{j,k}\delta_{\ell_1,\ell_2}, \qquad 
\]
where $\delta_{a,b} = 1$ if $a = b$ and $0$ otherwise. Xu \cite{Xu2001} gives more details regarding orthogonal polynomials on simplices and cubature formulae.

\subsubsection{Summability}
Summability of the orthonormal polynomials is important for explicitly defining the reproducing kernel Hilbert space, as we discuss later. Let us define,
\[P_\ell(x,y) := \sum_{k=1}^{r_\ell^d} P_{\ell,k}(x) P_{\ell,k}(y) = [P_\ell(x)]^T[P_\ell(y)],
\]
where $P_\ell = \left(P_{\ell,1}, \ldots, P_{\ell,r_\ell^d}\right)^T$. Note that, any other orthonormal basis of $\mathcal{V}_\ell^d$ can be obtained by multiplying an orthogonal matrix to $P_\ell$. Thus, $P_\ell(x,y)$ is independent of the choice of the orthonormal basis. Xu \cite{Xu1998} gives a closed form for this sum which we will use in our proofs. Theorem 2.2 of \cite{Xu1998} states that
\[
P_\ell(x,y) = \frac{2\ell + d}{(2\pi)^{d+1}d}  
\int_{[-1,1]^{d+1}} C_{2\ell}^{(d)}\left(\sum_{i=1}^{d+1} \sqrt{x_iy_i}t_i\right) \prod_{i=1}^{d+1}(1 - t_i^2)^{ -\frac{1}{2}} \df \bst,
\]
where $x_{d+1} = 1 - |x|, y_{d+1} = 1 - |y|$, and $C_{n}^{(\lambda)}$ is the Gegenbauer polynomial of degree $n = 2\ell$ and parameter $\lambda = d$.   More details about the Gegenbauer polynomial can be found in \cite{dunkl2001orthogonal,  Olver:2010:NHMF, reimer1990constructive}.

\subsubsection{Eigenfunctions}
Consider a partial differential operator $\nabla$ defined as
\[
\nabla := \sum_{i=1}^d x_i(1-x_i) \frac{\partial^2}{\partial x_i^2} - 2 \sum_{1 \leq i \leq j \leq d} x_i x_j \frac{\partial^2}{\partial x_i \partial x_j} + \sum_{i=1}^d( 1 - (d + 1)x_i) \frac{\partial }{\partial x_i}.
\]
It is well known (see \cite{Aktas2012}) that for all $P \in \mathcal{V}_\ell^d$, we have 
\[
\begin{split}
\nabla P = -\ell(\ell + d )P.
\end{split}
\]
In other words, orthogonal polynomials of degree $\ell$ are eigenfunctions of the second order partial differential operator $\nabla$ with eigenvalue $-\ell(\ell + d)$. Hence for our orthonormal basis $\{P_{\ell,k}\}$, we get for all $k \in \{1, \ldots, r_\ell^d\}$,
\[ (-\nabla)P_{\ell,k} = \ell(\ell + d)P_{\ell,k}.
\]
For any $r > 0$, we define the (pseudo-differential) operator $(-\nabla)^{\frac{r}{2}}$ by
\[
(-\nabla)^{\frac{r}{2}}P_{\ell,k} := [\ell(\ell+ d)]^{\frac{r}{2}}P_{\ell,k}.
\]
Since $-\nabla$ is a second order partial differential operator, we can intuitively think of $(-\nabla)^{\frac{r}{2}}$ as the $r$-th derivative operator.

\begin{remark}
 Throughout this Section, we work with the inner product on $L^2(T^d)$ as defined in (\ref{innerproduct}), where the integration is with respect to the constant weight of $d!$. There also exists a classical weight function on the $d$-dimensional simplex $T^d$ defined as in \cite{Xu2001} by
\[W_{\boldsymbol{\alpha}}(x) := w_{\boldsymbol{\alpha}} x_1^{\alpha_1} \cdots x_d^{\alpha_d}(1 - |x|)^{\alpha_{d+1}}
\qquad \text{ for }\;\; \alpha_1,\ldots,\alpha_d,\alpha_{d+1} > -1,\]
where $w_{\boldsymbol{\alpha}}$ is such that $\int_{T^d} W_{\boldsymbol{\alpha}}(x) \df x = 1$. With this weight function, the corresponding inner product can be defined as,
\begin{align*}
 \langle f,g \rangle_{\boldsymbol{\alpha}} = \int_{T^d} f(x) g(x) W_{\boldsymbol{\alpha}}(x) \df x.
\end{align*}
Choosing $\boldsymbol{\alpha} = \boldsymbol{0}$, we get our specific case. Most of the results in this Section can be generalized to arbitrary $\boldsymbol{\alpha}$. For more details on the general weight function and orthogonal polynomials, see \cite{Aktas2012, dunkl2001orthogonal, Xu2001}.
\end{remark}

In the following lemmas we summarize the properties of the function $P_\ell(x,y)$ that we will need in our proofs.

\begin{lemma} 
\label{uprbound}
For $\ell \geq 1$ and $x,y \in T^d$, we have $P_\ell(x,y) \leq M\ell^{2d}$, where $M$ is a constant depending only on $d$.
\end{lemma}
\begin{proof}
We have
\begin{equation}
\label{pneqn}
P_\ell(x,y) = \frac{2\ell + d}{(2\pi)^{d+1}d} 
\int_{[-1,1]^{d+1}} C_{2\ell}^{(d)}\left(\sum_{i=1}^{d+1} \sqrt{x_iy_i}t_i\right) \prod_{i=1}^{d+1}(1 - t_i^2)^{ -\frac{1}{2}} \df \bst.
\end{equation}
Since $x,y \in T^d$, we have $ -1 \leq \sum_{i=1}^{d+1}{\sqrt{x_iy_i}t_i} \leq 1$. The Gegenbauer polynomials have the following properties (see \cite{dunkl2001orthogonal}) :
\begin{align}
\label{genen_prop}
 \left|C_{\ell}^{(\lambda)}(x)\right| \leq C_\ell^{(\lambda)}(1)\;\; \text{ for } x \in [-1,1], \qquad C_\ell^{(\lambda)}(-x) = (-1)^\ell C_\ell^{(\lambda)}(x).
\end{align}
Thus, we get
\[\begin{split} 
P_\ell(x,y) &\leq \frac{(2\ell + d)C_{2\ell}^{(d)}(1)}{(2\pi)^{d+1}d} 
\int_{[-1,1]^{d+1}} \prod_{i=1}^{d+1}(1 - t_i^2)^{ -\frac{1}{2}} \df \bst\\
&= \frac{(2\ell + d)C_{2\ell}^{(d)}(1)}{(2\pi)^{d+1}d} 
\prod_{i=1}^{d+1} \int_{[-1,1]} (1 - t_i^2)^{ -\frac{1}{2}} \df t_i\\
&=  \frac{(2\ell + d)C_{2\ell}^{(d)}(1)}{2^{d+1}d},
\end{split}
\]
where the last equality follows from $ \int_{[-1,1]} (1 - t_i^2)^{ -\frac{1}{2}} \df t_i = \pi$. Now from \cite{dunkl2001orthogonal} we know $C_{2\ell}^{(d)}(1) = \binom{2\ell + 2d - 1}{2\ell}$. Thus, we have
\[ 
\begin{split}
P_\ell(x,y) &\leq \frac{2\ell + d}{2^{d+1}d}C_{2\ell}^{(d)}(1) = \frac{2\ell + d}{2^{d+1}d} \binom{2\ell + 2d - 1}{2\ell}\\
&= \frac{\left(2\ell + d\right)}{2^{d+1}d} \frac{(2\ell + 2d - 1) \ldots (2\ell + 1)}{(2d-1)!}\\
&\leq \frac{2\ell + d}{2^{d+1}d} \tilde{M} \ell^{2d-1} \leq M\ell^{2d},
\end{split}
\]
where $\tilde{M}$ and $M$ depend only on $d$.
\end{proof}

\begin{lemma} 
\label{ortho}
For $\ell \geq 1$ and for any $ x \in T^d$,
\[\int_{T^d} P_\ell(x,y) \df y = 0.
\]
\end{lemma}
\begin{proof}
Note that
\[
\begin{split}
\int_{T^d}P_\ell(x,y)\df y &= \int_{T^d} \sum_{k=1}^{r_\ell^d} P_{\ell,k}(x) P_{\ell,k}(y) \df y \\
&=  \sum_{k=1}^{r_\ell^d} P_{\ell,k}(x) \int_{T^d}  P_{\ell,k}(y) \df y.
\end{split}
\]
Now, for $\ell = 0, r_0^d = 1$. Thus, $P_{0,1}(x) = \frac{\partial^{|\boldsymbol{0}|}}{\partial x^{\boldsymbol{0}}} \left[ x^{\boldsymbol{0}} (1 - |x|)^{|\boldsymbol{0}|} \right] = 1$. Using this, we can write the above integral as
\[
\begin{split}
\int_{T^d}P_\ell(x,y)\df y &=  \sum_{k=1}^{r_\ell^d} P_{\ell,k}(x) \int_{T^d}P_{0,1}(y) P_{\ell,k}(y) \df y = 0.
\end{split}
\]
The last equality follows from the fact that on $T^d$ the basis functions $\{P_{\ell,k}\}$ are orthonormal with respect to the constant weight function.
\end{proof}

 With these definitions and notations we can now define our Sobolev spaces.

\subsection{The Sobolev Space $H_1$}
Let $f \in L_2(T^d)$.  The Fourier expansion of $f$ with respect to the family of orthonormal polynomials $\{P_{\ell,k} : 1\leq k \leq r_\ell^d, \ell = 0,1,2,\ldots \}$ is given by
\begin{equation}
\label{fourier}
 f(x) \sim \sum_{\ell=0}^\infty \sum_{k = 1}^{r_\ell^d} a_k^\ell(f) P_{\ell,k}(x), \qquad \textnormal{where} \qquad a_k^\ell(f) = d!\int_{T^d}f(x)P_{\ell,k}(x) \df x.
\end{equation} 
We denote by $\Pi^d$ the set of polynomials in $d$ variables on $T^d$ and by $\Pi_\ell^d$ the subset of polynomials of degree at most $\ell$, that is $\Pi^d = \cup_{\ell=0}^\infty \Pi_\ell^d$. For $r > 0$ and $\gamma > 0$, we define the Sobolev space $H_1 := H_{1,\gamma}^{(d,r)} \subseteq L_2(T^d)$ as the closure of $\Pi^d$ with respect to the norm
\[
||f||_1^2 := I_1(f)^2 + \frac{1}{\gamma}|| (-\nabla)^{\frac{r}{2}}f||_{L_2(T^d)}^2,
\]
where $I_1(f) = d! \int_{T^d} f(x) \df x$ and $(-\nabla)^{\frac{r}{2}}f(x) =  \sum_{\ell=0}^\infty\sum_{k=1}^{r_\ell^d}[\ell(\ell+ d)]^{\frac{r}{2}}a_k^\ell(f)P_{\ell,k}(x)$. Thus, it follows from Parseval's Theorem that for $r > 0$ and $f \in H_1$,
\[\begin{split}
||f||_1^2 &= a_1^0(f)^2 + \frac{1}{\gamma} \sum_{\ell=1}^\infty\sum_{k=1}^{r_\ell^d}[\ell(\ell+d)]^{r}a_k^\ell(f)^2\\
&=  \sum_{\ell=0}^\infty\sum_{k=1}^{r_\ell^d}B_{d,r,\gamma}(\ell)a_k^\ell(f)^2 <\infty,
\end{split}
\]
where 
\[B_{d,r,\gamma}(\ell) = \begin{cases}
1 & \textnormal{ if } \ell = 0\\
\frac{1}{\gamma}[\ell(\ell+d)]^r & \textnormal{ if } \ell \ge 1.
\end{cases}
\]
Thus, for any two functions $f,g \in H_1$ we can define the inner product as
\[ \langle f, g \rangle_1 := \sum_{\ell=0}^\infty\sum_{k=1}^{r_\ell^d}B_{d,r,\gamma}(\ell)a_k^\ell(f)a_k^\ell(g). 
\]
We now give a condition for the convergence of the Fourier series (\ref{fourier}), which is different from the condition in Kuo and Sloan \cite{Kuo2005}. The uniform convergence result for the Fourier series expansion on the simplex are more restrictive than for the unit sphere.

\begin{lemma}\label{uniform}
Let $f \in H_1$. Then the Fourier series of $f$ given in (\ref{fourier}) converges uniformly if $r >  d+1$.
\end{lemma}
\begin{proof}
We shall show that the Weierstrass' M-test holds for $r > d+1$ and thereby the series $\sum_{\ell=0}^{\infty} f_\ell(x)$ converges uniformly on $T^d$, where $f_\ell(x) = \sum_{k=1}^{r_\ell^d}a_k^\ell(f) P_{\ell,k}(x)$. To apply the Weierstrass' M-test, we first find a bound $M_\ell$ such that $|f_\ell(x)| \leq M_\ell$ for all $\ell \geq 0$ and all $x \in T^d$. Observe that for $\ell \geq 1$,
\begin{align*}
|f_\ell(x)| &= \left|\sum_{k=1}^{r_\ell^d}a_k^\ell(f) P_{\ell,k}(x)\right|\\
&= \left|\sum_{k=1}^{r_\ell^d}a_k^\ell(f) B_{d,r,\gamma}(\ell) ^{1/2} P_{\ell,k}(x) B_{d,r,\gamma}(\ell) ^{-1/2}\right|\\
&\leq \left(\sum_{k=1}^{r_\ell^d}a_k^\ell(f)^2 B_{d,r,\gamma}(\ell)\right)^{1/2}\left(\sum_{k=1}^{r_\ell^d}\frac{P_{\ell,k}(x)P_{\ell,k}(x)}
{B_{d,r,\gamma}(\ell)}\right)^{1/2}\\
&\leq ||f||_1 \left(\frac{P_\ell(x,x)}{B_{d,r,\gamma}(\ell)}\right)^{1/2} \leq ||f||_1 \sqrt{M} \frac{\ell^d}{\sqrt{B_{d,r,\gamma}(\ell)}} =: M_\ell
\end{align*}
by the Cauchy-Schwarz inequality; we have

 $$\sum_{k=1}^{r_\ell^d}a_k^\ell(f)^2 B_{d,r,\gamma}(\ell) \leq \sum_{\ell=0}^{\infty}\sum_{k=1}^{r_\ell^d}a_k^\ell(f)^2 B_{d,r,\gamma}(\ell) = ||f||_1^2 ;$$ and finally the last inequality follows from Lemma \ref{uprbound}.
For $\ell = 0$, we have
\[ |f_0(x)| = |a_1^0(f)P_{0,1}(x)| = |a_1^0(f)| = |I_1(f)| =: M_0
\]
Now,
\[\begin{split} 
\sum_{\ell=0}^{\infty}M_\ell &= M_0 + \sum_{\ell=1}^{\infty}M_\ell  = |I_1(f)| + ||f||_1\sqrt{M} \sum_{\ell=1}^\infty \frac{\ell^d}{\sqrt{B_{d,r,\gamma}(\ell)}}\\
&=  |I_1(f)| + ||f||_1\sqrt{M \gamma} \sum_{\ell=1}^\infty \frac{\ell^d}{[\ell(\ell+d)]^{r/2}}\\
& \leq |I_1(f)| + ||f||_1\sqrt{M \gamma} \sum_{\ell=1}^\infty \ell^{d-r}.
\end{split}
\] 
Thus, $\sum_{\ell=0}^{\infty}M_\ell < \infty$ if $ r > d+1$. Hence, if $r > d+1$, then Weierstrass' M-test holds and $\sum_{\ell=0}^{\infty}\sum_{k=1}^{r_\ell^d}a_k^\ell(f) P_{\ell,k}(x)$ converges uniformly on $T^d$.
\end{proof}

Using the result of Lemma \ref{uniform}, it easy to see the following embedding result for the Sobolev space $H_1$.
\begin{lemma}
\label{embedding}
If $r > d+1$, $H_1 \subseteq \mathcal{C}(T^d)$, where $\mathcal{C}(T^d)$ is the class of continuous functions on $T^d$.
\end{lemma}
\begin{proof}
Let $f \in H_1$. Since $r > d+1$, by Lemma \ref{uniform} we see that the Fourier series of $f$ converges uniformly. Now, as uniform convergence preserves continuity we have that $f$ is continuous and pointwise equal to its Fourier series. Hence, the result follows.
\end{proof}

\subsection{Reproducing Kernel Hilbert Space}
We will now show that $H_1$ is indeed a reproducing kernel Hilbert space. (For more details on RKHS see \cite{aronszajn1950theory}.)  Using the embedding result of Lemma \ref{embedding}, we can prove the following result. 
\begin{lemma}
Let $r > d+1$. For $f \in H_1$, we have $|f(x)| \leq C ||f||_1$ for some $C > 0$.
\end{lemma}
\begin{proof}
Since $f \in H_1 \subseteq \mathcal{C}(T^d)$ we get
\[\begin{split}
|f(x)| &= \left|\sum_{\ell=0}^{\infty}\sum_{k=1}^{r_\ell^d}a_k^\ell(f) P_{\ell,k}(x)\right|\\
&= \left|\sum_{\ell=0}^{\infty}\sum_{k=1}^{r_\ell^d}a_k^\ell(f) B_{d,r,\gamma}(\ell) ^{1/2} P_{\ell,k}(x) B_{d,r,\gamma}(\ell) ^{-1/2}\right|\\
&\leq \left(\sum_{\ell=0}^{\infty}\sum_{k=1}^{r_\ell^d}a_k^\ell(f)^2 B_{d,r,\gamma}(\ell)\right)^{1/2}\left(\sum_{\ell=0}^{\infty}\sum_{k=1}^{r_\ell^d}\frac{P_{\ell,k}(x)P_{\ell,k}(x)}{B_{d,r,\gamma}(\ell)}\right)^{1/2}\\
&= ||f||_1 \left(\sum_{\ell=0}^{\infty} \frac{P_\ell(x,x)}{B_{d,r,\gamma}(\ell)}\right)^{1/2} = ||f||_1 \left( 1 + \gamma\sum_{\ell=1}^{\infty} \frac{P_\ell(x,x)}{[\ell(\ell+d)]^r}\right)^{1/2}. 
\end{split}
\]
Now, 
\begin{equation}
\label{cdr}
 \sum_{\ell=1}^{\infty} \frac{P_\ell(x,x)}{[\ell(\ell+d)]^r} \leq M \sum_{\ell=1}^{\infty} \frac{\ell^{2d}}{[\ell(\ell+d)]^r} =: c_{d,r},
\end{equation}
 where the inequality follows from Lemma \ref{uprbound} and the convergence holds because $r > d+1$.
Thus,
\[
|f(x)|  \leq (1 + \gamma c_{d,r})^{1/2} ||f||_1.
\]
Taking $C = (1 + \gamma c_{d,r})^{1/2}$, we have our result. 
\end{proof}

The embedding result yields that point evaluation is a bounded linear functional on $H_1$ for $r > d+1$. Therefore, $H_1$ is a reproducing kernel Hilbert space. Thus, there exists a kernel $K_1 : T^d \times T^d \rightarrow \mathbb{R}$, which has the reproducing property; namely $K_1(x,y) = K_1(y,x)$ for all $x,y \in T^d$, $K_1(x, \cdot) \in H_1$ for all $x \in T^d$, and 
\[\langle f, K_1(\cdot,y) \rangle_1 = \sum_{\ell=0}^\infty\sum_{k=1}^{r_\ell^d} a_k^\ell(f) P_{\ell,k}(y) = f(y) \qquad \forall \; f \in H_1, y \in T^d.
\]
It is easy to see that the kernel $K_1(x,y)$ can be explicitly written as
\[\begin{split}
K_1(x,y) = K_{1,\gamma}^{(d,r)}(x,y) &= \sum_{\ell=0}^{\infty}\sum_{k=1}^{r_\ell^d}\frac{P_{\ell,k}(x)P_{\ell,k}(y)}{B_{d,r,\gamma}(\ell)} = 1 + \gamma \sum_{\ell=1}^{\infty} \frac{P_\ell(x,y)}{[\ell(\ell+d)]^r}.
\end{split}
\]
Note further that $\sum_{\ell=1}^{\infty} \frac{P_\ell(x,y)}{[\ell(\ell+d)]^r}$ converges uniformly for $r > d+1$. In fact from Lemma \ref{lemma_lower_bound} (proved later) we can show it converges absolutely for all $x, y \in T^d$. Thus $K_1(x,y) \leq 1+ \gamma c_{d,r}$, where $c_{d,r}$ is given in (\ref{cdr}) and $K_1$ is a continuous function on $T^d \times T^d$. 

\subsection{The Sobolev Space $H_m$}
From now onwards we shall assume that $r > d+1$. We closely follow \cite{Kuo2005} when defining the weighted Sobolev space $H_m = H_{m,\gamma_m}^{(d,r)}$ over the product of simplices $(T^d)^m$ as the tensor product
\[H_m = H_{m,\gamma_m}^{(d,r)} := H_{1,\gamma_{m,1}}^{(d,r)} \otimes H_{1,\gamma_{m,2}}^{(d,r)} \otimes \cdots \otimes H_{1,\gamma_{m,m}}^{(d,r)},
\]
where the weights in $\gamma_m = (\gamma_{m,1}, \ldots, \gamma_{m,m})$ are assumed to be positive and uniformly bounded, i.e., 
\begin{align}
\label{gamma_star}
 \gamma^* := \sup_{m \geq 1} \max_{1 \leq j \leq m} \gamma_{m,j} < \infty.
\end{align}
Any function $f \in H_m$ can be expressed as
\[f(\bsx) = f(x_1, \ldots, x_m) = \sum_{\bsl \in \mathbb{N}_0^m} \sum_{\bsk \in \mathcal{K}(m,\bsl)} a_{\bsk}^{\bsl}(f) \prod_{j=1}^m P_{\ell_j,k_j}(x_j),
\]
where $ \mathcal{K}(m,\bsl) := \{ \bsk \in \mathbb{N}^m : 1\leq k_j \leq r_{\ell_j}^d$ for each $j = 1, 2, \ldots, m\}$ and
\[
a_{\bsk}^{\bsl}(f) := (d!)^m \int_{(T^d)^m} f(\bsx) \prod_{j=1}^m P_{\ell_j,k_j}(x_j) \df \bsx.
\]
Similar to the space $H_1$, we can define the inner product on $H_m$ as
\[\langle f,g \rangle_m := \sum_{\bsl \in \mathbb{N}_0^m} \sum_{\bsk \in \mathcal{K}(m,\bsl)}\left(\prod_{j=1}^m B_{d,r,\gamma_{m,j}}(\ell_j) a_{\bsk}^{\bsl}(f) a_{\bsk}^{\bsl}(g)\right)
\]
and the reproducing kernel as
\[
\begin{split}
K_m(\bsx,\bsy) &:= \prod_{j=1}^m K_{1,\gamma_{m,j}}^{(d,r)}(x_j, y_j) = \prod_{j=1}^m \left( 1 + \gamma_{m,j} \sum_{\ell=1}^{\infty} \frac{P_\ell(x_j,y_j)}{[\ell(\ell+d)]^r} \right).
\end{split}
\]
\section{QMC tractability}
Kuo and Sloan \cite{Kuo2005} gave expressions for the worst-case error in reproducing kernel Hilbert spaces on products of spheres, in terms of the reproducing kernel. Here we state analogous results considering the domain to be products of simplices. The worst-case error using the QMC rule with points $(\bst_1, \ldots, \bst_n)$ is 
\[
\begin{split}
e_{n,m}^2 &= \sup_{f \in H_m, ||f||_m \leq 1} |I_m(f) - Q_{n,m}(f)|^2\\
&= (d!)^{2m} \int_{(T^d)^m}\int_{(T^d)^m} K_m(\bsx,\bsy) \df \bsx \df \bsy - \frac{2(d!)^m}{n} \sum_{i=1}^n \int_{(T^d)^m} K_m(\bsx,\bst_i) \df \bsx\\
& \;\;\;\; + \frac{1}{n^2} \sum_{i=1}^n \sum_{h=1}^n K_m(\bst_i,\bst_h).
\end{split}
\]
The initial error satisfies
\[e_{0,m}^2 = (d!)^{2m} \int_{(T^d)^m}\int_{(T^d)^m} K_m(\bsx,\bsy) \df \bsx \df \bsy.
\]
\label{sec4}
The mean square worst-case error over all cubature points is
\[ \begin{split}
\mathbb{E}(e_{n,m}^2) &:= (d!)^{nm} \int_{(T^d)m}\cdots \int_{(T^d)^m} e_{n,m}^2 (\bst_1, \ldots, \bst_n) \df \bst_1 \cdots \df \bst_n\\
&= \frac{1}{n} \left((d!)^m \int_{(T^d)^m} K_m(\bsx,\bsx) \df \bsx  - (d!)^{2m} \int_{(T^d)^m}\int_{(T^d)^m} K_m(\bsx,\bsy) \df \bsx \df \bsy
 \right).
\end{split}
\]
We now focus on our Sobolev space $H_m$ and give an upper and lower bound on $e_{n,m}^2$ which will then be used to prove QMC tractability in the space $H_m$. We begin with a lemma which is similar to Lemma 1 in \cite{Kuo2005}.

\begin{lemma}
\label{lemmaenm}
Let $r > d+1$. Then $e_{0,m} = 1$ and
\[e_{n,m}^2 = -1 +  \frac{1}{n^2} \sum_{i=1}^n \sum_{h=1}^n \prod_{j=1}^m \left( 1 + \gamma_{m,j} \sum_{\ell=1}^{\infty} \frac{P_\ell(t_{i,j},t_{h,j})}{[\ell(\ell+d)]^r}\right).
\] 
\end{lemma}
\begin{proof}
Observe that
\[\begin{split}
(d!)^m \int_{(T^d)^m} K_m(\bsx,\bsy) \df \bsx &= (d!)^m \int_{(T^d)^m}  \prod_{j=1}^m \left( 1 + \gamma_{m,j} \sum_{\ell=1}^{\infty} \frac{P_\ell(x_j,y_j)}{[\ell(\ell+d)]^r}\right) \df \bsx\\
&=\prod_{j=1}^m d! \int_{T^d}\left( 1 + \gamma_{m,j} \sum_{\ell=1}^{\infty} \frac{P_\ell(x_j,y_j)}{[\ell(\ell+d)]^r}\right) \df x_j\\
&=\prod_{j=1}^m \left( 1 + \gamma_{m,j} d! \int_{T^d}\sum_{\ell=1}^{\infty} \frac{P_\ell(x_j,y_j)}{[\ell(\ell+d)]^r}\df x_j\right). 
\end{split}
\]
Now, since $\sum_{\ell=1}^{\infty} \frac{P_\ell(x_j,y_j)}{[\ell(\ell+d)]^r}$ is uniformly bounded, by the bounded convergence theorem, we can interchange the sum and the integral. Thus, we get
\[
\begin{split}
(d!)^m \int_{(T^d)^m} K_m(\bsx,\bsy) \df \bsx &= \prod_{j=1}^m \left( 1 + \gamma_{m,j} \sum_{\ell=1}^{\infty} \frac{d! \int_{T^d}P_\ell(x_j,y_j)\df x_j}{[\ell(\ell+d)]^r}\right) \\
&= 1,
\end{split}
\]
where the last equality follows from Lemma \ref{ortho}. Thus, we get
\[\begin{split}
e_{0,m}^2 &= (d!)^{2m} \int_{(T^d)^m}\int_{(T^d)^m} K_m(\bsx,\bsy) \df \bsx \df \bsy\\
&= (d!)^m \int_{(T^d)^m} 1 \df \bsy = 1.
\end{split}
\]
We also get
\[
\begin{split}
e_{n,m}^2 
&= (d!)^{2m} \int_{(T^d)^m}\int_{(T^d)^m} K_m(\bsx,\bsy) \df \bsx \df \bsy - \frac{2(d!)^m}{n} \sum_{i=1}^n \int_{(T^d)^m} K_m(\bsx,\bst_i) \df \bsx\\
& \;\;\;\; + \frac{1}{n^2} \sum_{i=1}^n \sum_{h=1}^n K_m(\bst_i,\bst_h)\\
& = 1 - 2 + \frac{1}{n^2} \sum_{i=1}^n \sum_{h=1}^n K_m(\bst_i,\bst_h)\\
& = -1 +  \frac{1}{n^2} \sum_{i=1}^n \sum_{h=1}^n \prod_{j=1}^m \left( 1 + \gamma_{m,j} \sum_{\ell=1}^{\infty} \frac{P_\ell(t_{i,j},t_{h,j})}{[\ell(\ell+d)]^r}\right).
\end{split}
\]
\end{proof}
\subsection{Upper Bound for $e_{n,m}^2$}
We use the expected worst-case error to obtain a particular set of $\bst_1, \ldots, \bst_n \in (T^d)^m$ such that $e_{n,m}^2(\bst_1, \ldots, \bst_n)$ has the required upper bound. The following lemma gives the result. It is of the same flavor as Lemma 2 in \cite{Kuo2005}.
\begin{lemma}
\label{upperbound}
Let $r > d+1$. There exist $\bst_1, \ldots, \bst_n \in (T^d)^m$ such that
\[e_{n,m}^2(\bst_1, \ldots, \bst_n) \leq \frac{1}{n} \left( \prod_{j=1}^m \left( 1 + \gamma_{m,j}c_{d,r}\right) -1 \right),
\]
where $c_{d,r}$ is given in (\ref{cdr}).
\end{lemma}
\begin{proof}
Note that $\mathbb{E}(e_{n,m}^2)$ is the mean of $e_{n,m}^2$ over all possible selection of cubature points. Thus, there exists a collection $(\bst_1, \ldots, \bst_n)$ such that $e_{n,m}^2(\bst_1, \ldots, \bst_n) \leq \mathbb{E}(e_{n,m}^2)$. Now, 
\[\begin{split}
 \mathbb{E}(e_{n,m}^2) &= \frac{1}{n} \left((d!)^m \int_{(T^d)^m} K_m(\bsx,\bsx) \df \bsx  - (d!)^{2m} \int_{(T^d)^m}\int_{(T^d)^m} K_m(\bsx,\bsy) \df \bsx \df \bsy\right)\\
&= \frac{1}{n}\left((d!)^m \int_{(T^d)^m} K_m(\bsx,\bsx) \df \bsx  - 1\right)\\
&= \frac{1}{n}\left( \prod_{j=1}^m \left( 1 + \gamma_{m,j} \sum_{\ell=1}^{\infty} \frac{d! \int_{T^d}P_\ell(x_j,x_j)\df x_j}{[\ell(\ell+d)]^r}\right) -1\right).
\end{split}
\]
Note that
\[
\sum_{\ell=1}^{\infty} \frac{d! \int_{T^d}P_\ell(x_j,x_j)\df x_j}{[\ell(\ell+d)]^r} \leq \sum_{\ell=1}^{\infty} \frac{d! \int_{T^d}M\ell^{2d}\df x_j}{[\ell(\ell+d)]^r} = M \sum_{\ell=1}^{\infty} \frac{\ell^{2d}}{[\ell(\ell+d)]^r} = c_{d,r},
\]
where the first inequality follows from Lemma \ref{uprbound}. Furthermore, for all $j \in \{1,\ldots,m\}$, $
P_\ell(x_j, x_j) = \sum_{k=1}^{r_\ell^d} (P_{\ell,k}(x_j))^2 \ge 0$. Thus, we get for all $j \in \{1,\ldots,m\}$
\[
1 \leq \left( 1 + \gamma_{m,j} \sum_{\ell=1}^{\infty} \frac{d! \int_{T^d}P_\ell(x_j,x_j)\df x_j}{[\ell(\ell+d)]^r}\right) \leq  1 + \gamma_{m,j} c_{d,r}.
\]
Plugging it back into $\mathbb{E}(e_{n,m}^2)$, we see that there exist $\bst_1, \ldots, \bst_n$ such that
\[e_{n,m}^2(\bst_1, \ldots, \bst_n) \leq \mathbb{E}(e_{n,m}^2) \leq  \frac{1}{n} \left( \prod_{j=1}^m \left( 1 + \gamma_{m,j}c_{d,r}\right) -1 \right).
\]
\end{proof}

\subsection{Lower Bound on $e_{n,m}^2$}
We begin this section with a few notation. We add an extra subscript on the previous notation to identify the weights with respect to which the Sobolev space is defined. Let $|| \cdot ||_{m,\boldsymbol{\gamma}_m}$, $e_{n,m,\boldsymbol{\gamma}_m}$ and $K_{m,\boldsymbol{\gamma}}(\cdot, \cdot)$ denote the norm, worst-case error and the kernel, respectively, in the space $H_{m,\boldsymbol{\gamma}_{m}}^{(d,r)}$.

We follow the argument in \cite{Kuo2005, Sloan2001} to obtain our lower bound. The argument would have been much simpler as in \cite{Sloan1998}, if $K_{m,\boldsymbol{\gamma}}(\bsx,\bsy) \geq 0$. However that may not always be the case, so we introduce $\boldsymbol{\eta}_{m} = (\eta_{m,1}, \ldots, \eta_{m,m}),$ a collection of positive weights such that $\eta_{m,j} \leq \gamma_{m,j}$ for all $j = 1, \ldots, m$ and later choose $\boldsymbol{\eta}$ such that $K_{m, \boldsymbol{\eta}}(\bsx,\bsy)$ is non-negative.   

Observe, that as $\eta_{m,j} \leq \gamma_{m,j}$ for all $j = 1, \ldots, m$, we have $B_{d,r,\gamma_{m,j}}(\ell) \leq B_{d,r,\eta_{m,j}}(\ell)$ for all $j = 1, \ldots, m$ which shows that $||f||_{m,\boldsymbol{\gamma}_m} \leq ||f||_{m,\boldsymbol{\eta}_m}$. This implies that the unit ball of $H_{m,\boldsymbol{\eta}_{m}}^{(d,r)}$ is contained in the unit ball of $H_{m,\boldsymbol{\gamma}_{m}}^{(d,r)}$. Thus, from Definition (\ref{defenm}), we get 
\[e_{n,m,\boldsymbol{\eta}_m}(\bst_1, \ldots, \bst_n) \leq e_{n,m,\boldsymbol{\gamma}_m}(\bst_1, \ldots, \bst_n).
\]
Thus, it is enough to obtain a lower bound for $e_{n,m,\boldsymbol{\eta}_m}$ to get  a lower bound for $e_{n,m,\boldsymbol{\gamma}_m}$. To do that we first need to find the appropriate weights $\boldsymbol{\eta}_m$.
For $x,y \in T^d$, consider the continuous function
\begin{equation}
\label{geqn}
 g(x,y) = \sum_{\ell=1}^{\infty} \frac{P_\ell(x,y)}{[\ell(\ell+d)]^r}.
\end{equation}
Let $g_{\min}$ and $g_{\max}$ denote the minimum and maximum of $g(x,y)$. Note that by Lemma \ref{uprbound}, $g_{\max}$ is finite because $r > d+1$ by assumption. We shall also show that $g_{\min} > -\infty$. In fact we produce an explicit lower bound for $g_{\min}$.
\begin{lemma}
\label{lemma_lower_bound}
Let $r > d+1$. For $x, y \in T^d$, let $g(x,y)$ denote the function in (\ref{geqn}). Then,
\[
g_{\min} =  \inf_{x,y \in T^d} g(x,y)  \geq -c_{d,r},
\]
where $c_{d,r}$ is given in (\ref{cdr}).
\end{lemma}

\begin{proof}
From (\ref{pneqn}), we have
\[\begin{split}
P_\ell(x,y) &= \frac{2\ell + d}{(2\pi)^{d+1}d} \times 
\int_{[-1,1]^{d+1}} C_{2\ell}^{(d)}\left(\sum_{i=1}^{d+1} \sqrt{x_iy_i}t_i\right) \prod_{i=1}^{d+1}(1 - t_i^2)^{ -\frac{1}{2}} \df \bst\\
& \geq \frac{2\ell + d}{2^{d+1}d} \min_{u \in [-1,1]} C_{2\ell}^{(d)}(u).
\end{split}
\]
Now, from \cite{dunkl2001orthogonal} we know that the Gegenbauer polynomials satisfy $|C_\ell^{(\lambda)}(u)| \leq C_\ell^{(\lambda)}(1)$ for all $ u \in [-1,1]$. Therefore, $
\min_{u \in [-1,1]} C_{2\ell}^{(d)}(u) \geq -C_{2\ell}^{(d)}(1)$. From this we get,
\[\begin{split}
g_{\min} &= \inf_{x,y \in T^d} \sum_{\ell=1}^{\infty} \frac{P_\ell(x,y)}{[\ell(\ell+d)]^r} \geq \sum_{\ell=1}^{\infty} \frac{2\ell + d}{2^{d+1}d(\ell(\ell+d))^r} \min_{u \in [-1,1]} C_{2\ell}^{(d)}(u)\\
& \geq -\sum_{\ell=1}^{\infty} \frac{2\ell + d}{2^{d+1}d(\ell(\ell+d))^r} C_{2\ell}^{(d)}(1)\\
& \geq -M\sum_{\ell=1}^{\infty} \frac{\ell^{2d}}{[\ell(\ell+d)]^r} = -c_{d,r} > -\infty,
\end{split}
\]
where the third inequality follows from the proof of Lemma \ref{uprbound}.
\end{proof}

Thus, we get $-c_{d,r} \leq g_{\min} \leq g_{\max} \leq c_{d,r}$. Now, we define
\begin{equation}
\label{bdr}
b_{d,r} := \min\left(1, \frac{1}{\gamma^* |g_{min}|}\right),
\end{equation}
where $\gamma^*$ is given in (\ref{gamma_star}) and set $\eta_{m,j} = b_{d,r} \gamma_{m,j}$ for each $j = 1, \ldots, m$. With this definition it is easy to show that $\boldsymbol{\eta}_m$ has all desired properties. The non-negativity of the kernel is proved as follows.
\begin{lemma}
\label{kpos}
Let $\boldsymbol{\eta}_m = b_{d,r}\boldsymbol{\gamma}_m$, where $b_{d,r}$ is defined in (\ref{bdr}). Then $K_{m, \boldsymbol{\eta}}(\bsx,\bsy) \geq 0$.
\end{lemma}
\begin{proof}
Note that for each $j \in \{1, \ldots, m\}$
\[
\begin{split}
1 + \eta_{m,j} g(x_j,y_j)  &= 1 + b_{d,r}\gamma_{m,j} g(x_j,y_j) \\
& \geq 1 + b_{d,r}\gamma_{m,j} g_{\min} \geq 1 - b_{d,r}\gamma_{m,j} |g_{\min}| \\
&\geq  1 - b_{d,r}\gamma^*|g_{\min}| \geq 0,
\end{split}
\]
where the last inequality follows from (\ref{bdr}). Thus, we get
\[
K_{m, \boldsymbol{\eta}}(\bsx,\bsy) = \prod_{j=1}^m \left( 1 + \eta_{m,j} g(x_j,y_j) \right)  \geq 0
\]
\end{proof}

Now we are at the stage to prove the lower bound for $e_{n,m}$. Analogous to (\ref{geqn}), for $x \in T^d$ we define
\begin{equation}
\label{geqn2}
\tilde{g}(x) :=  \sum_{\ell=1}^{\infty} \frac{P_\ell(x,x)}{[\ell(\ell+d)]^r}.
\end{equation}
Let  $\tilde{g}_{\min} = \inf_{x \in T^d} \tilde{g}(x)$. Since $P_\ell(x,x) \geq 0$ for all $x \in T^d$, we have $\tilde{g}_{\min} \geq 0$. Using $\tilde{g}_{\min}$, we give the lower bound for $e_{n,m}$ in the following lemma. It is similar to Lemma 3 in \cite{Kuo2005}.

\begin{lemma}
\label{lowerbound}
Let $r > d+1$. For all $(\bst_1, \ldots, \bst_n) \in (T^d)^m$, we have
\[e_{n,m}^2(\bst_1, \ldots, \bst_n) \geq -1 + \frac{1}{n} \prod_{j=1}^m \left(1 + b_{d,r}\tilde{g}_{\min}\gamma_{m,j}\right),
\]
where $b_{d,r}$ is given in (\ref{bdr}) and $\tilde{g}_{\min}$ is the minimum of the function $\tilde{g}$ given in (\ref{geqn2}).
\end{lemma}
\begin{proof}
We have already shown that $e_{n,m} = e_{n,m,\boldsymbol{\gamma}_m}(\bst_1, \ldots, \bst_n) \geq e_{n,m,\boldsymbol{\eta}_m}(\bst_1, \ldots, \bst_n)$, where $\boldsymbol{\eta}_m = b_{d,r}\boldsymbol{\gamma}_m$. Now, from Lemmas \ref{lemmaenm} and \ref{kpos}, we have 
\[\begin{split}
e_{n,m,\boldsymbol{\eta}_m}^2(\bst_1, \ldots, \bst_n) &=-1 +  \frac{1}{n^2} \sum_{i=1}^n \sum_{h=1}^n \prod_{j=1}^m \left( 1 + \eta_{m,j} g(t_{i,j},t_{h,j})\right)\\
& \geq -1 + \frac{1}{n^2} \sum_{i=1}^n \prod_{j=1}^m \left( 1 + \eta_{m,j} g(t_{i,j},t_{i,j})\right)\\
& \geq -1 +  \frac{1}{n^2} \sum_{i=1}^n \prod_{j=1}^m \left( 1 + \eta_{m,j} \tilde{g}_{\min}\right)\\
& = -1 +  \frac{1}{n} \prod_{j=1}^m \left( 1 + \eta_{m,j} \tilde{g}_{\min}\right) = -1 +  \frac{1}{n} \prod_{j=1}^m \left( 1 + b_{d,r}\tilde{g}_{\min}\gamma_{m,j} \right),
\end{split}
\]
where we get the first inequality by dropping the terms, where $h \neq i$.
\end{proof}
\subsection{Tractability}
Now we state and prove our main theorem on QMC tractability.
\begin{theorem}
Let $r > d+1$, and let $\boldsymbol{\gamma}_{m} = (\gamma_{m,1},\ldots, \gamma_{m,m})$ be a collection of positive and uniformly bounded weights; i.e.,  $\gamma^*$ given in (\ref{gamma_star}) is finite . Then
\begin{enumerate}
\item Multivariate integration is strong polynomial tractable in $H_{m, {\bf{\gamma}}_{m}}^{(d,r)}$ iff
\[\limsup_{m \rightarrow \infty} \sum_{j=1}^m \gamma_{m,j} < \infty.
\]
If the above is true, then the $\epsilon$-exponent of strong polynomial tractability is at most 2.
\item Multivariate integration is polynomial tractable in $H_{m, \gamma_{m,j}}^{(d,r)}$ iff
\[ \beta := \limsup_{m \rightarrow \infty} \frac{\sum_{j=1}^m \gamma_{m,j}}{\log(m+1)} < \infty.
\]
If the above is true, then the  $\epsilon$-exponent of polynomial tractability is at most 2 and the  $m$-exponent of polynomial tractability is at most $c_{d,r}\beta$, where $c_{d,r}$ is given by (\ref{cdr}).
\item Multivariate integration is weakly tractable in $H_{m, \gamma_{m,j}}^{(d,r)}$ iff
\[
\lim_{m \rightarrow \infty} \frac{\sum_{j=1}^m \gamma_{m,j}}{m} = 0.
\]
\end{enumerate}
\end{theorem}

\begin{proof}
To recall definitions, $n(\epsilon, m)$ is the smallest $n$ for which $\bst_1, \ldots, \bst_n$ exist such that $e_{n,m} \leq \epsilon e_{0,m}$. The multivariate integration problem is said to polynomial tractable iff there exist non-negative $C, p,$ and $q$ such that, $ n(\epsilon,m) \le C \epsilon^{-p} m^q$ for all $m = 1,2, \ldots $ and for all $\epsilon \in (0,1)$. The problem is said to be strong polynomial tractable if the above holds with $q = 0$. The problem is weakly tractable iff $\lim_{\epsilon^{-1} + m \rightarrow \infty} \frac{\log n(\epsilon, m)}{\epsilon^{-1} + m} = 0.$

Fix any  $\epsilon \in (0,1)$ and $m \geq 1$. We will first show the if part of all statements. From Lemmas \ref{lemmaenm} and \ref{upperbound}, we have
\begin{align}
\label{upper_bound1}
n(\epsilon,m) &\leq  \frac{1}{\epsilon^2} \prod_{j=1}^m \left( 1 + \gamma_{m,j}c_{d,r}\right) =  \frac{1}{\epsilon^2} \exp \left(\sum_{j=1}^m \log \left( 1 + \gamma_{m,j}c_{d,r}\right)\right) \nonumber \\
& \leq  \frac{1}{\epsilon^2} \exp \left(c_{d,r}\sum_{j=1}^m \gamma_{m,j}\right),
\end{align}
where the last inequality follows from the fact that $\log(1+x) \leq x$ for $x > 0$. Now, if $\limsup_{m \rightarrow \infty} \sum_{j=1}^m \gamma_{m,j} < \infty$, then (\ref{upper_bound1}) gives $n(\epsilon,m) \leq C \epsilon ^{-2}$ for some constant $C$. Since $\epsilon$ and $m$ were arbitrary, we have strong polynomial tractability with $\epsilon$-exponent at most 2. 

To show polynomial tractability we rewrite (\ref{upper_bound1}) as
\begin{align}
\label{upper_bound2}
n(\epsilon,m) \leq  \frac{1}{\epsilon^2} (m+1)^\frac{c_{d,r}\sum_{j=1}^m \gamma_{m,j}}{\log(m+1)}.
\end{align}
If $\beta := \limsup_{m \rightarrow \infty} \frac{\sum_{j=1}^m \gamma_{m,j}}{\log(m+1)} < \infty$, then for any $\delta > 0$ there exists $m_\delta \geq 1$ such that 
\[\frac{\sum_{j=1}^m \gamma_{m,j}}{\log(m+1)} \leq \beta + \delta\qquad \text{ for all } m \geq m_{\delta}.
\]
This combined with the upper bound in (\ref{upper_bound2}) yields $n(\epsilon, m) \leq \epsilon^{-2} (m+1)^{c_{d,r} (\beta + \delta)}$. Thus, there exists a constant $C_\delta$ such that for all $m \geq 1$, $n(\epsilon, m) \leq C_\delta\epsilon^{-2} m^{c_{d,r} (\beta + \delta)}$. Since this holds for arbitrary $\delta > 0$, we have polynomial tractability with $\epsilon$-exponent at most 2 and $m$-exponent at most $c_{d,r}\beta$.

To show weak tractability, note that from (\ref{upper_bound1})
\begin{align}
\label{weak_tract1}
\lim_{\epsilon^{-1} + m \rightarrow \infty} \frac{\log n(\epsilon, m)}{\epsilon^{-1} + m} \leq \lim_{\epsilon^{-1} + m \rightarrow \infty} \frac{c_{d,r}\sum_{j=1}^m \gamma_{m,j} + 2\log \epsilon^{-1}}{\epsilon^{-1} + m}.
\end{align}
For any value of $\epsilon$, if $m \rightarrow \infty$ in (\ref{weak_tract1}), then by the assumption of  $\lim_{m \rightarrow \infty} \frac{\sum_{j=1}^m \gamma_{m,j}}{m} = 0$, we see that the limit is 0. If $m$ is finite and $\epsilon^{-1} \rightarrow \infty$, then also the limit is 0. Thus, we have weak tractability if $\lim_{m \rightarrow \infty} \frac{\sum_{j=1}^m \gamma_{m,j}}{m} = 0$.

Now we prove the only if part of all statements. From Lemma \ref{lowerbound}, we have
\begin{align}
n(\epsilon,m) &\geq
 \frac{1}{1 + \epsilon^2}  \prod_{j=1}^m \left( 1 + b_{d,r}\tilde{g}_{\min}\gamma_{m,j} \right) \nonumber \\
&= \frac{1}{1 + \epsilon^2} \exp\left( \sum_{j=1}^m \log\left(1 + b_{d,r}\tilde{g}_{\min}\gamma_{m,j} \right)\right) 
\end{align}
Let us define $M_{d,r} := b_{d,r}\tilde{g}_{\min}\gamma^*$. Then $ b_{d,r}\tilde{g}_{\min}\gamma_{m,j} \leq M_{d,r}$ for all $j = 1, \ldots, m$. Now, for any $x \in (0, M_{d,r}]$, we have $\log (1 + x) \geq \alpha_{d,r} x$ with $\alpha_{d,r} =  \log(1 + M_{d,r})/M_{d,r}$. Using this, we have
\begin{align}
\label{lower_bound1}
n(\epsilon, m) \geq \frac{1}{1 + \epsilon^2} \exp\left( \alpha_{d,r}b_{d,r}\tilde{g}_{\min}\sum_{j=1}^m \gamma_{m,j} \right).
\end{align}
Note that since $\tilde{g}_{\min} \geq 0$, we have $ \alpha_{d,r}b_{d,r}\tilde{g}_{\min} \geq 0$. Thus, if $\limsup_{m \rightarrow \infty} \sum_{j=1}^m \gamma_{m,j} = \infty$, then the bound in (\ref{lower_bound1}) implies $n(\epsilon,m) \rightarrow \infty$ as $m \rightarrow \infty$ which contradicts strong polynomial tractability.

The argument for polynomial tractability is a bit more subtle. Similar to (\ref{upper_bound2}), we rewrite (\ref{lower_bound1}) as
\begin{align}
\label{lower_bound2}
n(\epsilon,m) \geq \frac{1}{1 + \epsilon^2} (m+1)^{\frac{\alpha_{d,r}b_{d,r}\tilde{g}_{\min}\sum_{j=1}^m \gamma_{m,j}}{\log(m+1)}}. 
\end{align}
Now if $\limsup_{m \rightarrow \infty} \frac{\sum_{j=1}^m \gamma_{m,j}}{\log(m+1)} = \infty$, then there exists a subsequence $\{m_k\}$ such that $ \frac{\sum_{j=1}^{m_k} \gamma_{m_k,j}}{\log(m_k+1)}$ increases to infinity as $k \rightarrow \infty$. Fix any non-negative $C, p$, and $q$. We shall show that there exists an $\epsilon$ and $m$ such that $n(\epsilon,m) > C\epsilon^{-p}m^q$. Choose $\epsilon = 1/2$. Now there exists a $K$ such that for all $k \geq K$
\[
\frac{4}{5}(m_k+1)^{\frac{\alpha_{d,r}b_{d,r}\tilde{g}_{\min}\sum_{j=1}^{m_k} \gamma_{m_k,j}}{\log(m_k+1)}} > C 2^p (m_k + 1)^q 
\]
Thus, we have $n(1/2, m_k) > C 2^p m_k^q$ for all $k \geq K$. Since the choice of $C, p$ and $q$ was arbitrary, we get a contradiction to polynomial tractability. 

To show the only if part for weak tractability, note that if $\lim_{m \rightarrow \infty} \frac{\sum_{j=1}^m \gamma_{m,j}}{m} \neq 0$, then $\limsup_{m \rightarrow \infty} \frac{\sum_{j=1}^m \gamma_{m,j}}{m} > 0$.
From (\ref{lower_bound1}), we have
\begin{align}
\label{weak_tract2}
\limsup_{\epsilon^{-1} + m \rightarrow \infty} \frac{\log n(\epsilon, m)}{\epsilon^{-1} + m} \geq \limsup_{\epsilon^{-1} + m \rightarrow \infty} \frac{\alpha_{d,r}b_{d,r}\tilde{g}_{\min}\sum_{j=1}^m \gamma_{m,j} + \log(1 + \epsilon^{2})^{-1}}{\epsilon^{-1} + m} > 0,
\end{align}
where the last inequality follows by fixing $\epsilon = \epsilon_0$ and taking $m \rightarrow \infty$. This gives a contradiction to weak tractability.

Thus, we have shown that $\limsup_{m \rightarrow \infty} \sum_{j=1}^m \gamma_{m,j} < \infty $, $\limsup_{m \rightarrow \infty} \frac{\sum_{j=1}^m \gamma_{m,j}}{\log(m+1)} < \infty$, and $\lim_{m \rightarrow \infty} \frac{\sum_{j=1}^m \gamma_{m,j}}{m} = 0$ are necessary and sufficient for strong polynomial, polynomial and weak tractability, respectively.  
\end{proof}

\section{Conclusion}
\label{sec5}
Following the approach in \cite{Kuo2005}, we have shown that there exists a sequence of QMC methods for integration over $m$ fold products of simplices with a Monte Carlo rate of convergence. The proof technique uses properties of weighted reproducing kernel Hilbert spaces and orthonormal polynomials defined on simplices. Since the proof is based on an averaging argument, it does not throw any light on a specific construction of such points which achieve the Monte Carlo rate. 

In the special case of a single triangle, $T^2$, Basu and Owen \cite{Basu2014} give explicit constructions with a higher rate of convergence than Monte Carlo. In \cite{Basu2015}, the same authors generalized a  scrambled net procedure to construct points on the product of $m$, $d$-dimensional spaces to obtain a variance rate of $O(n^{-1 -2/d} (\log n)^{m-1})$. Much more work is needed to remove the dependency on $m$. The non-constructive arguments in \cite{Kuo2005, Sloan1998} were followed by explicit constructions \cite{hesse2007component, kuo2002component, sloan2002step,  sloan2002component, sloan2002constructing}.  We hope similar results will follow from this paper as well.

\section*{Acknowledgments}
This work was supported by the U.S. National Science Foundation under grant DMS-1407397. I would like to sincerely thank Prof. Art Owen for his support and discussions. I would also like to thank Subhabrata Sen and Devleena Samanta for their comments, and the anonymous referees, whose constructive reviews have greatly improved the paper.


\section*{References}
\bibliographystyle{elsarticle-harv}

\bibliography{tract2}

\begin{thebibliography}{31}
\expandafter\ifx\csname natexlab\endcsname\relax\def\natexlab#1{#1}\fi
\expandafter\ifx\csname url\endcsname\relax
  \def\url#1{\texttt{#1}}\fi
\expandafter\ifx\csname urlprefix\endcsname\relax\def\urlprefix{URL }\fi

\bibitem[{Aktaş and Xu(2013)}]{Aktas2012}
Aktaş, R., Xu, Y., 2013. {Sobolev orthogonal polynomials on a simplex}. Int.
  Math. Res. Notices 2013~(13), 3087--3131.

\bibitem[{Aronszajn(1950)}]{aronszajn1950theory}
Aronszajn, N., 1950. Theory of reproducing kernels. Trans. Amer. Math. Soc.
  68~(3), 337--404.

\bibitem[{Basu and Owen(2015{\natexlab{a}})}]{Basu2015}
Basu, K., Owen, A., 2015{\natexlab{a}}. Scrambled geometric net integration
  over general product spaces. Submitted. arXiv:1503.02737.

\bibitem[{Basu and Owen(2015{\natexlab{b}})}]{Basu2014}
Basu, K., Owen, A.~B., 2015{\natexlab{b}}. {Low-discrepancy constructions in
  the triangle}. SIAM J. Numer. Anal. 53~(2), 743--761.

\bibitem[{Brandolini et~al.(2013)Brandolini, Colzani, Gigante, and
  Travaglini}]{Brandolini2013}
Brandolini, L., Colzani, L., Gigante, G., Travaglini, G., 2013. {A
  Koksma--Hlawka inequality for simplices}. In: Trends in Harmonic Analysis.
  Springer, pp. 33--46.

\bibitem[{{\relax DLMF}(2014)}]{NIST:DLMF}
{\relax DLMF}, 2014. {NIST Digital Library of Mathematical Functions}.
  http://dlmf.nist.gov/, Release 1.0.9 of 2014-08-29, online companion to
  \cite{Olver:2010:NHMF}.
\newline\urlprefix\url{http://dlmf.nist.gov/}

\bibitem[{Dunkl and Xu(2014)}]{dunkl2001orthogonal}
Dunkl, C.~F., Xu, Y., 2014. Orthogonal polynomials of several variables. Vol.
  155. Cambridge University Press, New York, NY.

\bibitem[{Hesse et~al.(2007)Hesse, Kuo, and Sloan}]{hesse2007component}
Hesse, K., Kuo, F.~Y., Sloan, I.~H., 2007. A component-by-component approach to
  efficient numerical integration over products of spheres. J. Complexity
  23~(1), 25--51.

\bibitem[{Hickernell and Wo{\'z}niakowski(2001)}]{hickernell2001tractability}
Hickernell, F.~J., Wo{\'z}niakowski, H., 2001. {Tractability of multivariate
  integration for periodic functions}. J. Complexity 17~(4), 660--682.

\bibitem[{Kuo and Joe(2002)}]{kuo2002component}
Kuo, F.~Y., Joe, S., 2002. Component-by-component construction of good lattice
  rules with a composite number of points. J. Complexity 18~(4), 943--976.

\bibitem[{Kuo and Sloan(2005)}]{Kuo2005}
Kuo, F.~Y., Sloan, I.~H., 2005. {Quasi-Monte Carlo methods can be efficient for
  integration over products of spheres}. J. Complexity 21~(2), 196--210.

\bibitem[{Niederreiter(1992)}]{Niederreiter1992}
Niederreiter, H., 1992. {Random Number Generation and Quasi-Monte Carlo
  Methods}. Society for Industrial and Applied Mathematics, Philadelphia.

\bibitem[{Novak et~al.(1997)Novak, Sloan, and
  Wo{\'z}niakowski}]{novak1997tractability}
Novak, E., Sloan, I.~H., Wo{\'z}niakowski, H., 1997. {Tractability of tensor
  product linear operators}. J. Complexity 13~(4), 387--418.

\bibitem[{Novak and Wo{\'z}niakowski(2008)}]{novak2008tractability}
Novak, E., Wo{\'z}niakowski, H., 2008. {Tractability of Multivariate Problems:
  Linear Information}. Vol.~1. European Mathematical Society Publishing House,
  Z{\"u}rich.

\bibitem[{Novak and Wo{\'z}niakowski(2010)}]{novak2010tractability}
Novak, E., Wo{\'z}niakowski, H., 2010. {Tractability of Multivariate Problems:
  Standard Information for Functionals}. Vol.~2. European Mathematical Society
  Publishing House, Z{\"u}rich.

\bibitem[{Novak and Wo{\'z}niakowski(2012)}]{novak2012tractability}
Novak, E., Wo{\'z}niakowski, H., 2012. {Tractability of Multivariate Problems:
  Standard Information for Operators}. Vol.~3. European Mathematical Society
  Publishing House, Z{\"u}rich.

\bibitem[{Olver et~al.(2010)Olver, Lozier, Boisvert, and
  Clark}]{Olver:2010:NHMF}
Olver, F.~W.~J., Lozier, D.~W., Boisvert, R.~F., Clark, C.~W. (Eds.), 2010.
  {NIST Handbook of Mathematical Functions}. Cambridge University Press, New
  York, NY, print companion to \cite{NIST:DLMF}.

\bibitem[{Pillards and Cools(2004)}]{pillards2004theoretical}
Pillards, T., Cools, R., 2004. A theoretical view on transforming
  low-discrepancy sequences from a cube to a simplex. Monte Carlo Meth. and
  Appl. 10~(3-4), 511--529.

\bibitem[{Pillards and Cools(2005)}]{Pillards2005}
Pillards, T., Cools, R., 2005. {Transforming low-discrepancy sequences from a
  cube to a simplex}. J. Comput. Appl. Math. 174~(1), 29--42.

\bibitem[{Reimer(1990)}]{reimer1990constructive}
Reimer, M., 1990. {Constructive Theory of Multivariate Functions: with an
  application to tomography}. BI Wissenschafts-verlag, Mannheim, Wien,
  Z{\"u}rich.

\bibitem[{Sloan et~al.(2002{\natexlab{a}})Sloan, Kuo, and Joe}]{sloan2002step}
Sloan, I., Kuo, F., Joe, S., 2002{\natexlab{a}}. {On the step-by-step
  construction of quasi--Monte Carlo integration rules that achieve strong
  tractability error bounds in weighted Sobolev spaces}. Math. Comput.
  71~(240), 1609--1640.

\bibitem[{Sloan and Reztsov(2002)}]{sloan2002component}
Sloan, I., Reztsov, A., 2002. Component-by-component construction of good
  lattice rules. Math. Comput. 71~(237), 263--273.

\bibitem[{Sloan et~al.(2002{\natexlab{b}})Sloan, Kuo, and
  Joe}]{sloan2002constructing}
Sloan, I.~H., Kuo, F.~Y., Joe, S., 2002{\natexlab{b}}. {Constructing randomly
  shifted lattice rules in weighted Sobolev spaces}. SIAM J. Numer. Anal.
  40~(5), 1650--1665.

\bibitem[{Sloan and Wo\'{z}niakowski(1998)}]{Sloan1998}
Sloan, I.~H., Wo\'{z}niakowski, H., 1998. {When are quasi-Monte carlo
  algorithms efficient for high dimensional integrals?} J. Complexity 14~(1),
  1--33.

\bibitem[{Sloan and Wo\'{z}niakowski(2001)}]{Sloan2001}
Sloan, I.~H., Wo\'{z}niakowski, H., 2001. {Tractability of multivariate
  integration for weighted Korobov classes}. J. Complexity 17~(4), 697--721.

\bibitem[{Sloan and Wo{\'z}niakowski(2002)}]{sloan2002tractability}
Sloan, I.~H., Wo{\'z}niakowski, H., 2002. {Tractability of integration in
  non-periodic and periodic weighted tensor product Hilbert spaces}. J.
  Complexity 18~(2), 479--499.

\bibitem[{Wo{\'z}niakowski(1994)}]{wozniakowski1994tractability}
Wo{\'z}niakowski, H., 1994. {Tractability and strong tractability of linear
  multivariate problems}. J. Complexity 10~(1), 96--128.

\bibitem[{Wo\'{z}niakowski(1994)}]{wozniakowski1994tractability2}
Wo\'{z}niakowski, H., 1994. {Tractability and strong tractability of
  multivariate tensor product problems}. J. Comput. Inform 4~(1), 19.

\bibitem[{Xu(1998)}]{Xu1998}
Xu, Y., 1998. {Summability of Fourier orthogonal series for Jacobi weight
  functions on the simplex in $\mathbb{R}^{d}$}. P. Am. Math. Soc. 126~(10),
  3027--3036.

\bibitem[{Xu(2001)}]{Xu2001}
Xu, Y., 2001. {Orthogonal polynomials and cubature formulae on balls,
  simplices, and spheres}. J. Comput. Appl. Math. 127~(1-2), 349--368.

\bibitem[{Xu(2006)}]{xu2006analysis}
Xu, Y., 2006. {Analysis on the unit ball and on the simplex}. Electron. T.
  Numer Ana. 25, 284--301.

\end{thebibliography}

\end{document}